\documentclass{ps}
\usepackage{amsmath,amsfonts}
\usepackage{color}
\usepackage[applemac]{inputenc}
\usepackage{mathrsfs}
\usepackage{mathtools}
\usepackage{hyperref}

\newcommand{\R}{{\mathbb R}}

\newcommand{\N}{{\mathbb N}}

\newcommand{\E}{{\mathbb E}}
\newcommand{\PP}{{\mathbb P}}

\newcommand{\ind}{{\textbf{1}}}

\theoremstyle{plain}
\newtheorem{theorem}{Theorem}[section]
\newtheorem{lemma}[theorem]{Lemma}

\theoremstyle{definition}
\newtheorem{remark}{Remark}[section]

\newcommand{\mysection}{\setcounter{equation}{0} \section}

\begin{document}

\title{Density stability for some L\'evy-driven Stochastic Differential Equations.}

\author{L. Huang
}\address{Higher School of Economics, Moscow, lhuang@hse.ru}

\date{\today}

\begin{abstract}
We consider a Stochastic Differential Equation driven by a L\'evy process whose L\'evy measure satisfy a tempered stable domination.
We study how a perturbation of the coefficients reflects on the density of the solution.
We quantify the distance between the densities in term of the proximity of the coefficients.
This extend to the stable case the works of \cite{kona:kozh:meno:15}, where the noise is Gaussian.
\end{abstract}



\maketitle

\mysection{Introduction and Main Result}

This work is motivated by the recent article of Konakov \textit{et al.} \cite{kona:kozh:meno:15}, where the authors studied the sensitivity of the density of a diffusion process with respect to a perturbation on the coefficients.
In this paper, we aim to extend their result, formulated for the Brownian motion to other types of noises.
Specifically to the class of L\'evy processes whose L\'evy measure satisfy the domination \textbf{[H-1]} below.
For such L\'evy process $Z$, we consider the following SDE:
\begin{equation}\label{EDS1}
dX_t = b(t,X_t)dt + \sigma(t,X_{t^-})dZ_t.
\end{equation}
We are interested in a perturbation of this equation:
\begin{equation}\label{EDS1N}
dX_t^n = b_n(t,X_t^n)dt + \sigma_n(t,X_{t^-}^n)dZ_t.
\end{equation}
We assume that the coefficients $b$ and $\sigma$ can be obtained as a limit: $b_n \rightarrow b$ and $\sigma_n \rightarrow \sigma$, and we aim to control the distance between the densities of $X_t$ and $X_t^n$.
Quantifying such a distance can be useful for applicative purposes. See \cite{kona:kozh:meno:15} for a review of the literature.

The idea is to exploit the parametrix series to obtain explicit representations for the densities.
The expansions can in turn be linked to the coefficients of the SDE.
This allows us to control the distance between the densities using the distance between the coefficients.
Let us mention that the parametrix representation proves to be very useful when tracking the dependency of the density in some parameter of the coefficients.
For instance, we mention the works of Bally and Kohatsu-Higa \cite{ball:koha:14} who discuss regularity of the density with respect to the coefficients of the SDE.

The parametrix technique for SDEs has been studied in many occurrences. 
Let us mention the works of Kolokoltsov \cite{kolo:00} who first studied the stable case, more recent works of
Knopova and Kulik \cite{knop:kuli:14,knop:kuli:15}, and Huang \cite{huan:15}.
In this paper, we aim to extend the result of Konakov \textit{et al.} formulated for a Brownian noise to the setting of \cite{huan:15}.


We denote by \textbf{[H]} the following set of assumptions.
\begin{trivlist}
\item[\textbf{[H-1]}] $(Z_t)_{t\ge 0}$ is a symmetric L\'evy process without Gaussian part.
We denote by $\nu$ its L\'evy measure.
There is a non increasing function $\bar{q}: \R_+ \rightarrow \R_+$, $\mu$ a bounded measure on $S^{d-1}$, and $\alpha\in(0,2)$ such that:
\begin{eqnarray}
\nu(A) \le \int_{S^{d-1}} \int_0^{+\infty} \ind_{A}(s\theta) \frac{\bar q(s)}{s^{1+\alpha}}ds \mu(d\theta) = m(A). 
\end{eqnarray}
We assume one of the following:
\begin{itemize}
\item[\textbf{[H-1a]}] $\mu$ has a density with respect to the Lebesgue measure on the sphere.
\item[\textbf{[H-1b]}] there exists $\gamma\in[1,d]$ such that $\mu \left(B(\theta,r) \cap S^{d-1} \right) \le C r^{\gamma -1},$
with $\gamma+\alpha>d$, and for all $s>0$, there exists $C>0$ such that $\bar{q}(s) \le C \bar{q}(2s)$.
\end{itemize}

\item[\textbf{[H-2]}] 
Denoting by $\varphi_Z$ the L\'evy-Kintchine exponent of $(Z_t)_{t\ge 0}$,
there is $K>0$ such that :
\begin{equation}
\E\left( e^{i\langle p,Z_t \rangle}\right) = e^{t\varphi_Z(p)} \le e^{-Kt|p|^\alpha}, \ |p| >1.
\end{equation}

\item[\textbf{[H-3]}] $b, b_n: \R_+\times\R^d \rightarrow \R^d$ are measurable and bounded and
$\sigma,\sigma_n: \R_+\times \R^d \rightarrow \R^d\otimes\R^d$ are bounded and $\eta$-H\"older continuous $\eta \in (0,1)$.
Also, when $\alpha \le 1$, we impose $b,b_n=0$.

\item[\textbf{[H-4]}]  $\sigma,\sigma_n$ are uniformly elliptic. For all $x,\xi \in \R^d$, there exists $\kappa>1$ such that:
\begin{equation}
\kappa^{-1} |\xi|^2 \le \langle \xi, \sigma(t,x) \xi \rangle \le \kappa |\xi|^2, \ \ \ \kappa^{-1} |\xi|^2 \le \langle \xi, \sigma_n(t,x) \xi \rangle \le \kappa |\xi|^2.
\end{equation}

\item[\textbf{[H-5]}]  For all $A \in \mathcal{B}$, Borelian, we define the measure:
\begin{equation}
\nu_t(x,A)= \nu\big( \{z \in \R^d; \ \sigma(t,x)z \in A \}\big).
\end{equation}

We assume these measures to be H\"older continuous with respect to the first parameter, that is, 
for all $ A \in \mathcal{B}$, 
$$
|\nu_t(x,A) - \nu_t(x',A)|+ |\nu^{n}_t(x,A) - \nu^{n}_t(x',A)|  \le C \delta\wedge|x-x'|^{\eta(\alpha\wedge 1)} m(A).
$$

\item[\textbf{[H-6]}] We assume the following stability conditions. There is a sequence $\Delta_n \underset{n \rightarrow +\infty}{\longrightarrow} 0$ such that
\begin{itemize}
\item Stability for the L\'evy measures
for all $A \in \mathcal{B}$, $t\le T$,
$$
\sup_{x \in \R^{d}} |\nu_t(x,A) - \nu_t^n(x,A)| + \sup_{x\neq x'} \frac{|\nu_t(x,A) - \nu_t(x',A) - \nu_t^n(x,A) +\nu_t^n(x',A)|}{\delta \wedge|x-x'|^{\eta(\alpha\wedge 1)}} \le \Delta_n m(A).
$$

\item Stability for the drift coefficient: 
$$\sup_{(t,x)\in \R_+\times \R^d }|b(t,x) - b_n(t,x)| \le C \Delta_n.$$
\end{itemize}

\end{trivlist}

\begin{remark}
The assumptions \textbf{[H-5]} and \textbf{[H-6]} can seem restrictive.
However, in the case where $\sigma$ is real and a general L\'evy measure
this assumption follows from the control in H\"older norm formulated in \cite{kona:kozh:meno:15}.
\end{remark}



It has been shown in \cite{huan:15} that under assumptions \textbf{[H-1]} to \textbf{[H-5]}, the parametrix representation for the densities of 
$(X_t)_{t\ge 0}$ and $(X_t^n)_{t\ge 0}$ holds. We will use the following notations throughout this paper:
\begin{eqnarray*}
p(t,s,x,z)dz = \PP( X_s \in dz | X_t = x), \ p_n(t,s,x,z)dz = \PP( X^n_s \in dz | X^n_t = x).
\end{eqnarray*}

Moreover, the following density bound holds:
$$
p(t,T,x,y) + p_n(t,T,x,y) \le C \bar{p}(t,T,x,y):= \frac{(T-t)^{-d/\alpha}}{\left(1+ \frac{|y-x|}{(T-t)^{1/\alpha}} \right)^{\gamma+\alpha}} Q(|y-x|),
$$
where when \textbf{[H-1a]} holds, we take $Q=\bar{q}$ and set $\gamma=d$
and when \textbf{[H-1b]} holds, we take $\bar{Q}(\rho) = \min(1, \rho^{\gamma-1}) \bar{q}(\rho)$. 
The aim of this paper is to quantify the distance between $p$ and $p_n$ using $\bar{p}$ and $\Delta_n$.
Our main result is the following:

\begin{theorem}\label{main_result}
Fix a finite time horizon $T>0$. There exists $C>0$, for all $t\le T$ and all $x,y \in \R^d$,  such that
$$|(p-p_n)(t,T,x,y)| \le C \Delta_n \bar{p}(t,T,x,y).$$
\end{theorem}

The rest of this paper is organised as follows.
In Section \ref{The parametrix Setting}, we recall the background on the parametrix technique.
In Section \ref{Proof of the Main Result}, we prove the main result, using technical results that we prove in Section \ref{Preliminary results}.

\begin{remark}
The bounded drift assumption comes from the fact that in the Lipschitz case, the density estimate for the solution of the SDE involves the flow of the ODE associated to the considered SDE. In our case, we would have to compare the flows $\theta_{t,T}(y)$ and $\theta_{t,T}^n(y)$, respectively solutions of the ODEs:
$$
\frac{dx_t}{dt}= b(t,x_t), \ \frac{dx^n_t}{dt}= b_n(t,x^n_t), \ x_T = x_T^n=y.
$$
Thus, a comparison of the two flows has to be done.
Let us mention a similar procedure developed in \cite{dela:meno:10}, where the initial flow is compared to its linearization.

\end{remark}

\mysection{The parametrix Setting}\label{The parametrix Setting}

In this Section, we recall some facts about the parametrix representation.
The reader may consult Kolokoltsov \cite{kolo:00} or Huang \cite{huan:15} for a more extensive presentation.
The idea is to set up a parametrix for both the original SDE $(X_t)_{t\ge 0}$ and the perturbed one $(X_t^n)_{t\ge 0}$, and compare the two representations.
To that end, we define two frozen processes.
Fix a terminal position $y\in \R^d$ and consider:
$$
\tilde{X}_s = x + \int_t^s \sigma(u,y) dZu, \ \tilde{X}_s^n = x + \int_t^s \sigma_n(u,y) dZu.
$$

We omit the dependencies $\tilde{X}_s= \tilde{X}_s(t,x,y)$ and $\tilde{X}_s^n= \tilde{X}_s^n(t,x,y)$ in order to simplify the notations.
Formally, $\tilde{X}_s$ (resp. $\tilde{X}_s^n$) is the candidate to approximate $X_s$ (resp. $X_s^n$).
From assumption \textbf{[H-2]}, it is clear that the frozen processes have densities.
Those densities depend on the frozen parameter $y$, and we write $\tilde{p}(t,T,x,y)=\tilde{p}^y(t,T,x,y)$ (resp.
$\tilde{p}_n(t,T,x,y)=\tilde{p}_n^y(t,T,x,y)$).
The following estimate is proved in \cite{huan:15}.
\begin{lemma}
Fix a terminal time $T>0$. There exist $C>0$ such that for all $x,y\in \R^d$:
$$
\tilde{p}(t,T,x,y) + \tilde{p}_n(t,T,x,y) \le C \bar{p}(t,T,x,y).
$$
\end{lemma}

This estimate is established by comparing the frozen process to the driving noise. 
We know from Sztonyk \cite{szto:10} that under \textbf{[H]}, density estimates holds for the driving process $(Z_t)_{\ge 0}$.
Then, we can transfer the estimates on the frozen process thanks to the uniform ellipticity of $\sigma, \sigma_n$. See \cite{huan:15} for details.

The second step in the parametrix technique is to link the frozen process to the initial SDE. To that end, we define the integro-differential operator:
$$
L_t(z,\nabla_x)\varphi(x) = \langle b(t,z) ,\nabla_x \varphi(x) \rangle + \int_{\R^d} \varphi(x +\sigma(t,z) \xi) - \varphi(x) - \langle \nabla_x \varphi(x) ,\sigma(t,z) \xi\rangle \ind_{\{|\xi| \le 1\}} \nu(d\xi).
$$
Observe that when $x\in \R^d$ is the initial condition of \eqref{EDS1}, then $L_t(x,\nabla_x)$ is the generator of $(X_t)_{t\ge 0}$ and for $y\in\R^d$ a fixed terminal point, $L_t(y,\nabla_x)$ is the generator of the frozen process.
With obvious notations, we also denote $L_t^n(z,\nabla_x)$ the generators of the perturbed equation, with $b_n$ and $\sigma_n$ instead of $b$ and $\sigma$.
We have the following result.

\begin{theorem}

Under assumptions \textbf{[H]}, $X_T$ (resp. $X_T^n$) has a density with respect to the Lebesgue measure, and the following representation holds:
$$
p(t,T,x,y) = \sum_{k \ge 0} \tilde{p}\otimes H^{(k)}(t,T,x,y), \ p_n(t,T,x,y) = \sum_{k \ge 0} \tilde{p}_n\otimes H_n^{(k)}(t,T,x,y), 
$$
where $H(t,T,x,y) = \Big(L_t(x,\nabla_x) - L_t(y,\nabla_x)\Big)\tilde{p}(t,T,x,y)$ and $H_n(t,T,x,y) = \Big(L^n_t(x,\nabla_x) - L^n_t(y,\nabla_x)\Big)\tilde{p}_n(t,T,x,y)$.
Besides, $\otimes$ denote the space time convolution:
$$
\varphi \otimes \psi (t,T,x,y) = \int_t^Tdu \int_{\R^d} \varphi(t,u,x,z) \psi (u,T,z,y) dz, \ \varphi^{(k)}(t,T,x,y) = \varphi^{(k-1)}\otimes \varphi(t,T,x,y),
$$
and $\varphi^{(0)}= Id$.
Finally, the following density bound holds:
$$
p(t,T,x,y)+p _n(t,T,x,y) \le C \bar{p}(t,T,x,y).
$$
\end{theorem}

We refer to Huang \cite{huan:15} for the proof of this statement. 
The main idea of the proof is to prove a \textit{regularisation property} for $H$ that allows to compensate the singularities induced by the generators. We define:
\begin{eqnarray*}
\rho(t,T,x,y) = \delta \wedge |y-x|^{\eta(\alpha \wedge 1)}\bar{p}(t,T,x,y),\ \bar{H}(t,T,x,y) = \frac{\delta \wedge |y-x|^{\eta(\alpha \wedge 1)}}{T-t} \bar{p}(t,T,x,y).
\end{eqnarray*}
It has been shown in \cite{huan:15} that: $|H(t,T,x,y)| + |H_n(t,T,x,y)| \le C \bar{H}(t,T,x,y)$.
Thus, to control the series, it suffices to control the iterated convolution of $\bar{p}$ and $\bar{H}$. 
We define 
\begin{eqnarray*}
\rho_{2k}(t,T,x,y) &=&    \frac{(T-t)^{k\omega}}{k! \omega^{2k}} \Big( (T-t)^{k\omega} \bar {p} + (\bar {p} + \rho) \Big)(t,T,x,y)\\
\rho_{2k+1}(t,T,x,y)&=&   \frac{(T-t)^{k\omega}}{(k+1)! \omega^{2k+1}} \Big( (T-t)^{(k+1)\omega} \bar {p}+ (T-t)^\omega(\bar {p} +\rho)+ \rho \Big)(t,T,x,y).
\end{eqnarray*}
The following estimate has been proved in \cite{huan:15}. For all $t \le T$, there exists $C, C_{}>1$ such that for all $x,y\in\R^d$:
\begin{eqnarray}\label{pxH}
|\tilde{p} \otimes H^{(m)}|(t,T,x,y)+ |\tilde{p}_n \otimes H_n^{(m)}|(t,T,x,y)\le C \bar{p}\otimes \bar{H}^{(m)}(t,T,x,y) \le C^{m} \rho_m(t,T,x,y).
\end{eqnarray}
Using the last estimate, we can prove that the parametrix series is absolutely convergent.
We can now use this representation to control the distance between $p$ and $p_n$, in the lines of Konakov \textit{et al.} \cite{kona:kozh:meno:15}. 

%

\mysection{Proof of the Main Result}\label{Proof of the Main Result}

In this section, we prove our main result. We rely on technical results whose proofs are postponed to Section \ref{Preliminary results}.
The idea is to compare the two series representations and get an estimate for the difference of each term involving $\Delta_n\bar{p}(t,T,x,y)$.
The proof is very similar to the proof of the convergence of the parametrix series.
Exploiting the parametrix representation for $p$ and $p_n$, we write:
\begin{eqnarray*}
|p(t,T,x,y) - p_n(t,T,x,y) |
&=& \left| \sum_{k=0}^{+\infty} \Big(\tilde{p}\otimes H^{(k)}-\tilde{p}_n\otimes H_n^{(k)} \Big)(t,T,x,y) \right|
\end{eqnarray*}
%

We now control the series term by term. 
We proceed by induction.
The first term above is the difference between the frozen densities.
Taking $\beta=0$ in Lemma  \ref{CTRL_DERIV_DENS} gives:
\begin{equation}\label{k=0}
|\tilde{p}(t,T,x,y)- \tilde{p}_n(t,T,x,y)| \le C \Delta_n \bar{p}(t,T,x,y).
\end{equation}

To control the next terms, observe that we can split for a general $m\in \N$:
\begin{eqnarray}
\tilde{p}\otimes H^{(m+1)} - \tilde{p}_n\otimes H_n^{(m+1)} &=& \tilde{p}\otimes H^{(m+1)} -\tilde{p}_n\otimes H_n^{(m)}\otimes H + \tilde{p}_n\otimes H_n^{(m)}\otimes H -\tilde{p}_n\otimes H_n^{(m+1)} \nonumber\\
&=& \Big( \tilde{p}\otimes H^{(m)} - \tilde{p}_n\otimes H_n^{(m)}\Big)\otimes H +\tilde{p}_n\otimes H_n^{(m)}\otimes (H-H_n). \label{DECOMP_TERM}
\end{eqnarray}

We treat the two terms separately.
First, we look of an estimate for $ \tilde{p}\otimes H^{(m)} - \tilde{p}_n\otimes H_n^{(m)}$.
We have the following result:
\begin{lemma}\label{Initialisation}
Fix $T>0$. There exists $C_{\ref{Initialisation}}>1$ such that for all $x,y \in\R^d$, $\forall 0 \le t \le T$,
\begin{eqnarray*}
\left| \Big(\tilde{p}\otimes H^{(m)}- \tilde{p}_n\otimes H_n^{(m)} \Big)\right| \le \Delta_n (C_{\ref{Initialisation}})^{m} \rho_m(t,T,x,y).
\end{eqnarray*}
\end{lemma}

\begin{proof}

We set $C_{\ref{Initialisation}}= C_{\ref{CTR_H-H_n}} \times C_0$, where $C_{\ref{CTR_H-H_n}}$ is the constant appearing in Lemma \ref{CTR_H-H_n}, and we proceed by induction.
\begin{trivlist}
\item[$(H_0)$:]The case $k=0$ follows from equation \eqref{k=0} since $\rho(t,T,x,y) >0$.

\item[$(H_{m} \Rightarrow H_{m+1})$:] We split:
\begin{eqnarray*}
\tilde{p}\otimes H^{(m+1)} - \tilde{p}_n\otimes H_n^{(m+1)} = \Big( \tilde{p}\otimes H^{(m)} - \tilde{p}_n\otimes H_n^{(m)}\Big)\otimes H +\tilde{p}_n\otimes H_n^{(m)}\otimes (H-H_n).
\end{eqnarray*}

Now we treat the two terms separately.
It has been shown in \cite{huan:15} that there exists $C_0>1$ such that:
\begin{eqnarray}\label{IND_RHOxH}
\rho_{m}\otimes \bar{H}(t,T,x,y) \le C_0 \rho_{m+1}(t,T,x,y),
\end{eqnarray}
thus we have:
\begin{eqnarray*}
\left|\Big(\tilde{p}\otimes H^{(m)} - \tilde{p}_n\otimes H_n^{(m)}\Big)\otimes H\right|(t,T,x,y) \le  \Delta_n(C_{\ref{Initialisation}})^{m}  \rho_{m} \otimes \bar{H}(t,T,x,y)\le\Delta_n  (C_{\ref{Initialisation}})^{m} C_0  \rho_{m+1}(t,T,x,y),
\end{eqnarray*}
the estimate follows from the fact that $C_0 \le C_{\ref{CTR_H-H_n}} \times C_0 =C_{\ref{Initialisation}}$.

For the second term, we use estimate \eqref{pxH}
and Lemma \ref{CTR_H-H_n}
to get
\begin{eqnarray*}
\Big|\tilde{p}_n\otimes H_n^{(m)}\otimes (H-H_n) \Big|(t,T,x,y) \le  \Delta_{n}(C_{\ref{Initialisation}})^{m}C_{\ref{CTR_H-H_n}} C_0  \rho_{m} \otimes \bar{H}(t,T,x,y)\le \Delta_{n}(C_{\ref{Initialisation}})^{m+1}  \rho_{m+1}(t,T,x,y).
\end{eqnarray*}
%
%
%
%
\end{trivlist}
\end{proof}

For the second part of the right hand side of \eqref{DECOMP_TERM}, we can use the estimate \eqref{pxH} to control $|\tilde{p}_n\otimes H_n^{(m)}|$ by $\rho_m$ and Lemma \ref{CTR_H-H_n} to control $|H-H_n|$ by $\Delta_n\bar{H}$.
We then estimate $\rho_{m}\otimes \bar{H} \le C \rho_{m+1}$, thus we obtain a summable bound with $\Delta_n$ in factor.
To complete the proof, it remains us point out that $\sum_{k\ge0}  (C_{\ref{Initialisation}})^{k} \rho_{k} < +\infty$, and that the sum of the series yields the right hand side in Theorem \ref{main_result}.


%

\mysection{Preliminary results}\label{Preliminary results}

In this Section, we prove the estimates used un the last section.
The proof our main result relies on an estimation of the difference between the two parametrix representations. 
Therefore, we have to control the difference of the frozen densities, and the kernels $H,H_n$.

\begin{lemma}\label{CTRL_DERIV_DENS}
Let $\Phi_t(x,\nabla_x)$ be a pseudo-differential operator, and denote by $\phi_t(x,p)$ its symbol.
Assume that $|\phi_t(x,p)| \le C |p|^{\beta}$.
There exists $C>0$ for all $t<T$, for all $x,y \in \R^d$, such that:
$$
\left|\Phi_t(x,\nabla_x)\Big(\tilde{p}- \tilde{p}_n \Big)(t,T,x,y)\right| \le C \frac{\Delta_n }{(T-t)^{\frac{|\beta|}{\alpha}}} \bar{p}(t,T,x,y).
$$
Note that taking $\beta=0$ yields estimate \eqref{k=0}.
\end{lemma}


To prove this Lemma, we rely on arguments developed in \cite{huan:15}. Keeping the notations,  
we exhibited how for fixed times, the frozen process can be linked to the marginals of a L\'evy process.
Fix $0<t<T$, and $x,y\in \R^d$. Then 
$$
\tilde{X}_T= x + \int_t^T \sigma(u,y)dZ_u \overset{(d)}{=} x + \mathcal{S}_{T-t},
$$
where $(\mathcal{S}_u)_{u\ge 0}$ is a L\'evy process satisfying assumptions \textbf{[H]}.
Precisely, let us denote by $\nu_{\mathcal{S}}$ the L\'evy measure of $\mathcal{S}$. There exists $\mu_{\mathcal{S}}$ such that
setting $\sigma_v= \sigma( (T-t)v+t , y)$, we have:
$$
\nu_{S}(A) = \int_0^1 \int_{\R^d} \ind_{\{ \sigma_v \xi \in A\}} \nu(d\xi) dv \le  \int_0^{+\infty} \int_{S^{d-1}} \ind_{\{ \rho \theta \in A\}} \frac{\bar{q}(\rho)}{\rho^{1+\alpha}} d\rho \mu_{\mathcal{S}}(d\theta). 
$$

Exploiting the L\'evy-It\^o decomposition of $\mathcal{S}_u= M_u + N_u$, with $M$ martingale and $N$ compound Poisson process, we can write:
$$
\tilde{p}(t,T,x,y) = \int_{\R^d} p_M(T-t,y-x-\xi) P_{N_{T-t}}(d\xi),
$$
where $p_M$ designates the density of the martingale and $P_{N_{T-t}}(d\xi)$ the law of the compound Poisson process.
The same decomposition holds for $\tilde{X}_{T}^n$.
Now, we point out that the operator $\Phi_t(x,\nabla_x)$ acts on the variable $x$, which is only present in the density of the martingale.
Consequently, we have:
\begin{eqnarray*}
\Phi_t(x,\nabla_x)\Big(\tilde{p}- \tilde{p}_n \Big)(t,T,x,y)
&=&  \int_{\R^d} \Phi_t(x,\nabla_x) \Big(p_M-p_{M^n} \Big)(T-t,y-x-\xi) P_{N_{T-t}}(d\xi) \\
&&+ \int_{\R^d} \Phi_t(x,\nabla_x) p_{M^n}(T-t,y-x-\xi) \Big(P_{N_{T-t}}(d\xi) -P_{N^n_{T-t}}(d\xi)  \Big).
\end{eqnarray*}

To prove Lemma  \ref{CTRL_DERIV_DENS}, we establish the following estimates.
We refer to the procedure developed in Sztonyk \cite{szto:10} to see how these estimates can be used to derive the final density estimate.
\begin{lemma}\label{operation_martingale}
Fix $m\ge 1$. There exists $C_m,C>0$, such that for all $t<T$, $\forall y,x,\xi \in \R^d$, 
\begin{eqnarray}\label{diff_dens_mart}
\left|\Phi_t(x,\nabla_x) \Big(p_M-p_{M^n} \Big)(T-t,y-x-\xi)\right| \le C_m \frac{\Delta_n}{(T-t)^{\beta/\alpha}} (T-t)^{-d/\alpha} \left( 1+ \frac{|y-x-\xi|}{(T-t)^{1/\alpha}}\right)^{-m},\\
\Big| P_{N_{T-t}}(d\xi) -P_{N^n_{T-t}}(d\xi) \Big| \le C \Delta_n \left( \sum_{k=0}^{+\infty} \frac{(T-t)^{k}}{k!} \bar{m}^{*k}(d\xi)+ \sum_{k=0}^{+\infty} \frac{(T-t)^k\bar{\nu}_{\mathcal{S}^n}^{*k}(dz) }{k!} \right),\label{diff_mes_pois}
\end{eqnarray}
where $\bar{m}(d\xi) = \ind_{\{|\xi| \ge (T-t)^{1/\alpha}\}}m(d\xi)$, and $\bar{m}^{*k}$ denotes the $k^{th}$ fold convolution of $\bar{m}$ with itself.
\end{lemma}


\begin{proof}[Proof of Lemma \ref{operation_martingale}]
We prove the first estimate by invariance of the Schwartz's space for the Fourier transform.
The Fourier transform of $M$ writes:
$$
\varphi_M(p) = \int_{\R^d} \Big(e^{i \langle p,\xi \rangle} - 1 - i \langle p, \xi \rangle \Big) \ind_{\{|\xi| \le (T-t)^{1/\alpha}\}} \nu_\mathcal{S}(d\xi).
$$
From \textbf{[H-2]}, this Fourier transform is integrable, thus we can express the density of $M$ as a Fourier inverse:
$$
p_M(T-t,y-x-\xi) = \frac{1}{(2\pi)^d} \int_{\R^d} e^{-i \langle p , y -x-\xi  \rangle }  \exp \Big( (T-t) \varphi_M(p) \Big) dp.
$$
Consequently, the action of the pseudo-differential operator $\Phi_t(x\nabla_x)$ writes:
$$
\Phi_t(x\nabla_x)p_M(T-t,y-x-\xi) = \frac{1}{(2\pi)^d} \int_{\R^d} e^{-i \langle p , y -x-\xi  \rangle } \phi_t(x,p)\exp \Big( (T-t) \varphi_M(p) \Big) dp.
$$
Since we have the same results for the density of $\tilde{X}^n_t$, when taking the difference, we can write using the mean value theorem:
%
\begin{eqnarray}
&&\Phi_t(x\nabla_x)\Big(p_M-p_{M^n} \Big)(T-t,y-x-\xi)\nonumber\\
&=&\int_0^1 d\lambda \frac{1}{(2\pi)^d} \int_{\R^d} dp e^{-i \langle p , y -x-\xi  \rangle }\left[ \phi_t(x,p)
 e^{ (T-t) \big(\lambda \varphi_M(p) + (1-\lambda) \varphi_{M^n}(p) \big) } (T-t)\Big(\varphi_M(p) - \varphi_{M^n}(p) \Big) \right]\nonumber\\
 &=&\frac{\Delta_n}{(T-t)^{\frac{\beta}{\alpha}}}\int_0^1 d\lambda \frac{1}{(2\pi)^d} \int_{\R^d} dp e^{-i \langle p , y -x-\xi  \rangle }\nonumber\\
&& \times \left[ \phi_t(x,p)
 e^{ (T-t) \big(\lambda \varphi_M(p) + (1-\lambda) \varphi_{M^n}(p) \big) } (T-t)^{1+\frac{\beta}{\alpha}}\frac{\varphi_M(p) - \varphi_{M^n}(p)}{\Delta_n} \right]\nonumber\\
 &=&\frac{\Delta_n}{(T-t)^{\frac{\beta}{\alpha}}}\int_0^1 d\lambda \frac{(T-t)^{-d/\alpha}}{(2\pi)^d} \int_{\R^d} dp e^{-i \langle q , \frac{y -x-\xi}{(T-t)^{1/\alpha}}  \rangle }  \label{goal_1}\\
&& \times \left[ \phi_t\left(x,\frac{q}{(T-t)^{1/\alpha}} \right)
 e^{ (T-t) \big(\lambda \varphi_M(\frac{q}{(T-t)^{1/\alpha}}) + (1-\lambda) \varphi_{M^n}(\frac{q}{(T-t)^{1/\alpha}}) \big) } (T-t)^{1+\frac{\gamma}{\alpha}}\frac{(\varphi_M - \varphi_{M^n})(\frac{q}{(T-t)^{1/\alpha}})}{\Delta_n} \right]\nonumber,
\end{eqnarray}
where to get to the last equality, we changed variables to $q= (T-t)^{1/\alpha}p$.
The announced estimate will now hold when we prove that uniformly in the parameters $\lambda,t,T,n$, the expression between brackets is in Schwartz's space.
In the $q$ variable, the above expression is smooth, thanks to the cut-off (see \cite{szto:10} and the references therein).
For the integrability, observe that from assumptions on $\phi_t(x,p)$, we can bound $\left|\phi_t\left(x,\frac{q}{(T-t)^{1/\alpha}}\right) \right| \le \frac{|q|^{\beta}}{(T-t)^{\beta/\alpha}}$.
On the other hand, we have for all $p\in \R^d$:
\begin{eqnarray*}
\lambda\varphi_M(p) + (1-\lambda) \varphi_{M^n}(p) &=& \int_{\R^d} \Big(e^{i \langle p,\xi \rangle} - 1 - i \langle p, \xi \rangle \Big) \ind_{\{|\xi| \le (T-t)^{1/\alpha}\}} \Big( \lambda\nu_\mathcal{S}(d\xi)+ (1-\lambda)\nu_{\mathcal{S}^n}(d\xi) \Big) \\
&\le& \int_{\R^d} \Big(e^{i \langle p,\rho \varsigma \rangle} - 1 - i \langle p, \rho \varsigma \rangle \Big) \ind_{\{\rho \le (T-t)^{1/\alpha}\}}
\frac{\bar{q}(\rho)}{\rho^{1+\alpha}}d\rho ( \lambda \mu_\mathcal{S} + (1-\lambda) \mu_{\mathcal{S}^n})
(d\varsigma).
\end{eqnarray*}

We take $p=q(T-t)^{-1/\alpha}$ and change variables to $r=\rho(T-t)^{-1/\alpha}$.
Recall that from the doubling condition, $\bar{q}((T-t)^{1/\alpha} r) \le \bar{q}(r)$, we get:
\begin{eqnarray*}
\lambda\varphi_M(p) + (1-\lambda) \varphi_{M^n}(p)
\le \frac{C}{T-t}\int_{\R^d} \Big(e^{i \langle q,r \varsigma \rangle} - 1 - i \langle q, r \varsigma \rangle \Big) \ind_{\{r \le 1\}}
\frac{\bar{q}(r)}{r^{1+\alpha}}d\rho ( \lambda \mu_\mathcal{S} + (1-\lambda) \mu_{\mathcal{S}^n})(d\varsigma) 
\le \frac{C|q|^\alpha}{T-t}.
\end{eqnarray*}

The last inequality comes from assumption \textbf{[H-2]}.
We now turn to the difference of the exponent $\varphi_{M}-\varphi_{M^n}$.
Similarly to the last computation, we deduce from \textbf{[H-6]} that measures $|\nu_\mathcal{S}(dz) -\nu_{\mathcal{S}^n}(dz)|\le \Delta_n m(dz)$, so that
\begin{eqnarray*}
(\varphi_{M}-\varphi_{M^n})\left( \frac{q}{(T-t)^{1/\alpha}}\right) &\le& 
\int_{\R^d} \Big(e^{i \langle q,z \rangle} - 1 - i \langle q, z \rangle \Big) \ind_{\{|z| \le (T-t)^{1/\alpha}\}}
( \nu_\mathcal{S} -\nu_{\mathcal{S}^n})(dz) \\
&\le& \frac{1}{T-t}\int_{\R_+\times S^{d-1}} \Big(e^{i \langle q,r \varsigma \rangle} - 1 - i \langle q, r \varsigma \rangle \Big) \ind_{\{r \le 1\}}
\frac{\bar{q}(r)}{r^{1+\alpha}}d\rho \mu(d\varsigma)
\end{eqnarray*}
where to get the last inequality, we changed variables to $r=(T-t)^{-1/\alpha}|z|$. Consequently, we can write:
$$
\left| \phi_t\left(x,\frac{q}{(T-t)^{1/\alpha}} \right)
 e^{ (T-t) \big(\lambda \varphi_M + (1-\lambda) \varphi_{M^n}\big)(\frac{q}{(T-t)^{1/\alpha}})  } (T-t)^{1+\frac{\beta}{\alpha}}\frac{(\varphi_M - \varphi_{M^n})(\frac{q}{(T-t)^{1/\alpha}})}{\Delta_n} \right| \le |q|^{\beta+\alpha}e^{-C |q|^{\alpha}}.
$$

Thus, the expression between brackets in \eqref{goal_1} is in Schwartz's space uniformly in the parameters $\lambda,t,T,n$, thus so is its Fourier inverse, and estimate \eqref{diff_dens_mart} follows.

We prove the next estimate by induction.
We denote $\bar{\nu}_{\mathcal{S}} (dz) = \ind_{\{|z| \le (T-t)^{1/\alpha}\}}\nu_{\mathcal{S}}(dz)$, 
the characteristic measure of the compound Poisson process $(N_u)_{u \ge 0}$.
We know  (see \cite{szto:10}) that the law of $(N_u)_{u \ge 0}$ actually writes $P_{N_u}=e^{-u \bar{\nu}_\mathcal S(\R^d)} \sum_{k=0}^{+\infty} \frac{u^k\bar{\nu}_{\mathcal S}^{*k}(dz) }{k!}$.
The same decomposition holds for $N_u^n$.
Consequently, when taking the difference:
$$
P_{N_{u}}(dz) - P_{N^n_{u}}(dz) = e^{-u \bar{\nu}_\mathcal S(\R^d)} \sum_{k=0}^{+\infty} \frac{u^k }{k!} \Big(\bar{\nu}_{\mathcal S}^{*k} - \bar{\nu}_{\mathcal{S}^n}^{*k} \Big)(dz)
+
\Big( e^{-u \bar{\nu}_\mathcal S(\R^d)} -e^{-u \bar{\nu}_{\mathcal{S}^n}(\R^d)} \Big)\sum_{k=0}^{+\infty} \frac{u^k\bar{\nu}_{\mathcal{S}^n}^{*k}(dz) }{k!}.
$$

From Taylor's formula, we deduce that the second term is controlled by
$\Delta_n\sum_{k=0}^{+\infty} \frac{u^k\bar{\nu}_{\mathcal{S}^n}^{*k}(dz) }{k!}$.
It remains us to prove that $|\bar{\nu}_{\mathcal S}^{*k}(dz)-\bar{\nu}_{\mathcal S^n}^{*k}(dz)| \le \Delta_n C^k\bar{m}^{*k}(dz)$, which can be done by a direct induction since from \textbf{[H-6]} and the definition of the measures $\nu_{\mathcal S},\nu_{\mathcal {S}^n}$, we have $|\nu_{\mathcal{S}^n}(dz)-\nu_{\mathcal{S}}(dz) |\le \Delta_n m(dz)$.
%
\end{proof}

We now turn to the estimate on the difference of the parametrix kernels $H-H_n$.
We have the following estimate:

\begin{lemma}\label{CTR_H-H_n}
There exists $C_{\ref{CTR_H-H_n}}>0$ for all $t<T$, for all $x,y\in \R^d$, such that:
$$
|H-H_n|(t,T,x,y)  \le C_{\ref{CTR_H-H_n}}\Delta_n\frac{\delta \wedge |y-x|^{\eta(\alpha\wedge1)}}{T-t}\bar{p}(t,T,x,y).
$$
\end{lemma}

By definition of the parametrix kernel, we can split:
\begin{eqnarray*}
(H-H_n)(t,T,x,y) &=& \Big(L_t(x,\nabla_x) - L_t( y, \nabla_x) \Big) \Big(\tilde{p}-\tilde{p}_n \Big)(t,T,x,y)\\
&&+ \Big(L_t(x,\nabla_x) - L_t( y, \nabla_x)-L_t^n(x,\nabla_x)+ L_t^n( y, \nabla_x)  \Big)\tilde{p}_n(t,T,x,y).
\end{eqnarray*}
The first contribution is controlled by Lemma \ref{CTRL_DERIV_DENS}, since denoting by $l_t(z,p),l_t^n(z,p)$ the symbols of $L_t(z,\nabla_x),L_t^n(z,\nabla_x)$ respectively, we have $|l_t(z,p)| + |l_t^n(z,p)| \le C |p|^{\alpha}$. Thus, we focus on the second contribution.

\begin{lemma}
There exists $C>0$ for all $t<T$, for all $x,y\in \R^d$, such that:
$$
\left|\Big(L_t(x,\nabla_x) - L_t( y, \nabla_x)-L_t^n(x,\nabla_x)+ L_t^n( y, \nabla_x)  \Big)\tilde{p}_n(t,T,x,y) \right| 
\le C \Delta_n \frac{ \delta\wedge|y-x|^{\eta(\alpha \wedge 1)}}{T-t} \bar{p}(t,T,x,y).
$$

\end{lemma}

\begin{proof}

We focus on the integro-differential part of the generator. For the gradient part, the estimate stems from the stability assumption on $b-b_n$ (see \textbf{[H-6]}).
We recall the definition of:
$$
\nu_t(x,A) = \nu (\{z\in \R^d; \sigma(t,x)z \in A \}), \ \nu^n_t(x,A) = \nu (\{z\in \R^d; \sigma_n(t,x)z \in A \}).
$$

With this definition, observe that from the assumption on the L\'evy measure, we can write:
\begin{eqnarray*}
&&\Big(L_t(x,\nabla_x) - L_t( y, \nabla_x)-L_t^n(x,\nabla_x)+ L_t^n( y, \nabla_x)  \Big)\tilde{p}_n(t,T,x,y)\\
&=&\int_{\R^d} \Big( \tilde{p}_n(t,T,x+z,y) - \tilde{p}_n(t,T,x,y) -\langle \nabla_x \tilde{p}_n(t,T,x,y), z\rangle \ind_{\{|z|\le C(T-t)^{1/\alpha}\} } \Big)\\
&&\times\Big( \nu_t(x,dz)  -\nu_t(y,dz)  -\nu_t^n(x,dz)  +\nu_t^n(y,dz) \Big)\\
&\le& \Delta_n \delta\wedge |x-y|^{\eta(\alpha\wedge 1)}  \mathcal{L} \tilde{p}_n(t,T,x,y),
\end{eqnarray*}

where we define:
$$
\mathcal{L} \tilde{p}_n(t,T,x,y)= \int_{\R^d} \Big( \tilde{p}_n(t,T,x+\rho \theta,y) - \tilde{p}_n(t,T,x,y) -\langle \nabla_x \tilde{p}_n(t,T,x,y), \rho \theta\rangle \ind_{\{ \rho \le C(T-t)^{1/\alpha}\} } \Big) \frac{\bar{q}(\rho)}{\rho^{1+\alpha}}\mu(d\theta).
$$

The proof follows from the upper bound $\mathcal{L} \tilde{p}_n(t,T,x,y)\le C  \frac{1}{T-t}\bar{p}(t,T,x,y)$.
%
%
%
%
From now, we follow the proof in \cite{huan:15}. 
The main idea is to split the operator in small and big jump contributions.
Let us define the following operators:
\begin{eqnarray*}
\mathcal{L}_M \varphi (x) &=& \int_0^{(T-t)^{1/\alpha}} \int_{S^{d-1}} \Big(\varphi( x + \rho \theta ) - \varphi(x) - \langle \nabla_x \varphi(x),\rho \theta \rangle \Big)  \bar{q}(\rho)\frac{d\rho}{\rho^{1+\alpha}} \mu(d\theta) \\
\mathcal{L}_N \varphi (x) &=&\int_{(T-t)^{1/\alpha}}^{+\infty} \int_{S^{d-1}} \Big(\varphi( x +\rho \theta ) - \varphi(x) \Big)\bar{q}(\rho)\frac{d\rho}{\rho^{1+\alpha}}\mu(d\theta).
\end{eqnarray*}

Observe that due to the symmetry of $\mu$, we can change the cut-off function so that the identity holds:
$$
\big(L_t(x,\nabla_x)-L_t(\theta_{t,T}(y),\nabla_x)\big)\varphi (x) \le \mathcal{L}_M \varphi (x) + \mathcal{L}_N \varphi (x).
$$

We now handle separately the small and big jumps.
Let us start with the small jumps.
We claim that:
\begin{equation}\label{ctrl_op_mart}
\mathcal{L}_M  \tilde{p}(t,T,x,y) \le C \frac{\delta \wedge |y-x|^{\eta ( \alpha \wedge 1)}}{T-t} \bar{p}(t,T,x,y).
\end{equation}

Recalling the notations of \cite{huan:15}, the frozen density decomposes as
\begin{equation}\label{DESINTEGRATION_DENS_LEVY_ITO}
\tilde{p}(t,T,x,y) =p_\mathcal{S}(T-t,y - x)= \int_{\R^d} p_M(T-t, y- x - \xi ) P_{N_{T-t}}(d\xi).
\end{equation}

Plugging this identity in the definition of the operator $\mathcal{L}_M$ yields:
\begin{eqnarray*} 
\mathcal{L}_M\tilde{p}(t,T,x,y) = \int_{\R^d}P_{N_{T-t}}(d\xi)
\int_0^{(T-t)^{1/\alpha}} \int_{S^{d-1}} \Big(p_M(T-t, y - x-\rho \theta - \xi ) - p_M(T-t, y - x - \xi )\\
- \langle \nabla_x p_M(T-t,y - x - \xi ),\rho\theta \rangle \Big)  \bar{q}(\rho)\frac{d\rho}{\rho^{1+\alpha}} \mu(d\theta).
\end{eqnarray*}

Observe that since the jumps are truncated, the density of the martingale $p_M$ is smooth.
Moreover, the successive derivatives of $p_M$ are be controlled as follows:
$$
\big|\nabla_x^\beta p_M(u,x)\big| \le C_m u^{-|\beta|/\alpha}u^{-d/\alpha} \left( 1+\frac{|x|}{u^{1/\alpha}}\right)^{-m}, \ \ \forall m \ge1.
$$
See the proof of Lemma 2 in \cite{szto:10} or \cite{huan:15}.
We now use Taylor's formula to expand under the integral:
\begin{eqnarray*}
&&p_M(T-t,y - x-\rho\theta - \xi ) - p_M(T-t, y - x - \xi ) - \langle \nabla_x p_M(T-t, y - x - \xi ),\rho\theta \rangle\\
&=& \frac12\nabla_x^2 p_M(T-t, y - x - \xi )\rho^2 \theta^{(2)} 
+\int_0^1  \nabla_x^3  p_M(T-t, y - x -(1-\lambda) \rho\theta - \xi ) \rho^3\theta^{(3)}\frac{\lambda^2}{2}d \lambda.
\end{eqnarray*}

Now, from the control on the derivative, we obtain for the main part:
\begin{eqnarray*}
&&\left| \int_{\R^d}P_{N_{T-t}}(d\xi) \int_0^{(T-t)^{1/\alpha}} \int_{S^{d-1}} \frac12\nabla_x^2 p_M(T-t,y - x - \xi )\rho^2 \theta^{(2)}  \bar{q}(\rho)\frac{d\rho}{\rho^{1+\alpha}} \mu(d\theta)\right|\\
&\le& C  \int_{\R^d}P_{N_{T-t}}(d\xi) \int_0^{(T-t)^{1/\alpha}} \int_{S^{d-1}} \left(\frac{\rho}{(T-t)^{1/\alpha}}\right)^2 
\frac{(T-t)^{-d/\alpha}}{ \left(1+ \frac{|y - x - \xi|}{(T-t)^{1/\alpha}} \right)^{-m}} \bar{q}(\rho)\frac{d\rho}{\rho^{1+\alpha}} \mu(d\theta)\\
&=&\int_{\R^d}P_{N_{T-t}}(d\xi)\frac{(T-t)^{-d/\alpha}}{ \left(1+ \frac{|y - x - \xi|}{(T-t)^{1/\alpha}} \right)^{-m}} 
\times  \underbrace{\int_0^{(T-t)^{1/\alpha}}  \left(\frac{\rho}{(T-t)^{1/\alpha}}\right)^2 \bar{q}(\rho)\frac{d\rho}{\rho^{1+\alpha}} }_{\le \frac{1}{T-t}}\times \mu(S^{d-1}).
\end{eqnarray*}

Finally, we recall that Sztonyk \cite{szto:10} proved that:
$\int_{\R^d}P_{N_{T-t}}(d\xi)\frac{(T-t)^{-d/\alpha}}{ \left(1+ \frac{|\theta_{t,T}(y) - x - \xi|}{(T-t)^{1/\alpha}} \right)^{m}} 
\le \bar{p}(t,T,x,y).$
%
%
Consequently, the bound \eqref{ctrl_op_mart} hold for the main term. We now turn to the remainder. 
Putting the absolute value inside the integrals allows us to interchange the order of integration.
We thus have:
\begin{eqnarray*}
\int_{S^{d-1}}\int_0^{(T-t)^{1/\alpha}}  \int_0^1 
\left(\int_{\R^d}P_{N_{T-t}}(d\xi)  \big| \nabla_x^3  p_M(T-t, y - x -(1-\lambda) \rho\theta - \xi )\big| \right)
 \rho^3 |\theta^{(3)} | \frac{\lambda^2}{2}d \lambda  \bar{q}(\rho)\frac{d\rho}{\rho^{1+\alpha}}\mu(d\theta).
\end{eqnarray*}

We now use the control on the derivative and write:
\begin{eqnarray*}
\int_{\R^d}P_{N_{T-t}}(d\xi)  \big| \nabla_x^3  p_M(T-t, y - x -(1-\lambda) \rho\theta - \xi )\big|
&\le& \frac{C}{(T-t)^{3/\alpha}}\int_{\R^d}P_{N_{T-t}}(d\xi)   \frac{(T-t)^{-d/\alpha}}{ \left(1+ \frac{|y - x -(1-\lambda) \rho\theta- \xi|}{(T-t)^{1/\alpha}} \right)^{-m}}\\
&\le& \frac{C}{(T-t)^{3/\alpha}} \frac{(T-t)^{-d/\alpha}}{\left( 1 + \frac{|x+(1-\lambda) \rho\theta -y|}{(T-t)^{1/\alpha}}\right)^{\alpha +\gamma}}\bar{q}(|x+(1-\lambda) \rho\theta -y|).
\end{eqnarray*}

Plugging the last inequality gives the upper bound for the remainder:
\begin{eqnarray*}
\int_{S^{d-1}}\int_0^{(T-t)^{1/\alpha}}  \int_0^1  \frac{(T-t)^{-d/\alpha}}{\left( 1 + \frac{|x+(1-\lambda) \rho\theta -y|}{(T-t)^{1/\alpha}}\right)^{\alpha +\gamma}}\bar{q}(|x+(1-\lambda) \rho\theta -y|)
\left( \frac{\rho}{(T-t)^{1/\alpha}}\right)^3\frac{\lambda^2}{2}d \lambda  \bar{q}(\rho)\frac{d\rho}{\rho^{1+\alpha}} \mu(d\theta).
\end{eqnarray*}

Now, we have to discuss according to the global regime of the density.
If $|x-y| \le \frac12 (1-\lambda) \rho$, or if $\frac12 (1-\lambda) \rho \le  |x-y| $, then 
$$
\frac{(T-t)^{-d/\alpha}}{\left( 1 + \frac{|x+(1-\lambda) \rho \theta -y|}{(T-t)^{1/\alpha}}\right)^{\alpha +\gamma}}\bar{q}(|x+(1-\lambda) \rho\theta -y|)
\le \frac{(T-t)^{-d/\alpha}}{\left( 1 + \frac{|x-y|}{(T-t)^{1/\alpha}}\right)^{\alpha +\gamma}}\bar{q}(|x-y|) \asymp \bar{p}(t,T,x,y).
$$

On the other hand, when $\frac12 |x-y|\le (1-\lambda) \rho \le 3/2 |x-y|$, since $\rho \le (T-t)^{1/\alpha}$, the diagonal regime holds and
$$
\frac{(T-t)^{-d/\alpha}}{\left( 1 + \frac{|x+(1-\lambda) \rho\theta -y|}{(T-t)^{1/\alpha}}\right)^{\alpha +\gamma}}\bar{q}(|x+(1-\lambda) \rho\theta -y|) \le (T-t)^{-d/\alpha} \asymp \bar{p}(t,T,x,y).
$$

Consequently, in both cases, we have:
$$
\frac{(T-t)^{-d/\alpha}}{\left( 1 + \frac{|x+(1-\lambda) \rho\theta -y|}{(T-t)^{1/\alpha}}\right)^{\alpha +\gamma}}\bar{q}(|x+(1-\lambda) \rho\theta -y|)  \le \bar{p}(t,T,x,y),
$$

and we thus have:
\begin{eqnarray*}
&&\int_{S^{d-1}}\int_0^{(T-t)^{1/\alpha}}  \int_0^1  \frac{(T-t)^{-d/\alpha}}{\left( 1 + \frac{|x+(1-\lambda) \rho\theta - y|}{(T-t)^{1/\alpha}}\right)^{\alpha +\gamma}}\bar{q}(|x+(1-\lambda) \rho\theta -y|)
\left( \frac{\rho}{(T-t)^{1/\alpha}}\right)^3\frac{\lambda^2}{2}d \lambda \bar{q}(\rho) \frac{d\rho}{\rho^{1+\alpha}} \mu(d\theta)\\
&\le& 
\bar{p}(t,T,x,y)\int_{S^{d-1}} \int_0^{(T-t)^{1/\alpha}}  \int_0^1  \left( \frac{\rho}{(T-t)^{1/\alpha}}\right)^3\frac{\lambda^2}{2}d \lambda  \bar{q}(\rho) \frac{d\rho}{\rho^{1+\alpha}} \mu(d\theta)\le  C \frac{1}{T-t} \bar{p}(t,T,x,y).
\end{eqnarray*}

We now turn to the large jumps.
Observe that for the large jumps, the L\'evy measure is not singular and we can write:
\begin{eqnarray*}
\left|\mathcal{L}_N \tilde{p}(t,T,x,y)\right| \le \left|\int_{(T-t)^{1/\alpha}}^{+\infty} \int_{S^{d-1}} \tilde{p}(t,T x +\rho \theta,y) \bar{q}(\rho)\frac{d\rho}{\rho^{1+\alpha}}\mu(d\theta) \right| + \tilde{p}(t,T,x,y) \underbrace{\int_{(T-t)^{1/\alpha}}^{+ \infty} \int_{S^{d-1}} \bar{q}(\rho)\frac{d\rho}{\rho^{1+\alpha}} \mu(d\theta)}_{\le C \frac{1}{T-t}}.
\end{eqnarray*}
We can thus focus on the remaining integral.
Assume first that the diagonal regime holds, that is $|x-y|\le (T-t)^{1/\alpha}$.
In that case, note that $\bar{p}(t,T,x,y) \asymp (T-t)^{1/\alpha}$.
Also, since we have the global estimate $ \tilde{p}(t,T,x+\rho\theta,y)\le C(T-t)^{-d/\alpha}$, we can take out the density and the integral yields the estimate $\big|\mathcal{L}_N \tilde{p}(t,T,x,y) \big| \le \frac{C}{T-t}\bar{p}(t,T,x,y)$.

Assume now that the off diagonal regime holds, that is $|y-x| \ge (T-t)^{1/\alpha}$.
The regime of $\tilde{p}(t,T,x+\rho\theta,y) $ is given by $|y-x-\rho\theta|$.
Thus, thanks to the triangle inequality, when $|y-x|\le 1/2 \rho$, or when $\rho \le 1/2|y-x|$,
the density $\tilde{p}(t,T,x+\rho\theta,y) $ is off-diagonal with $\tilde{p}(t,T,x+\rho\theta,y) \le C \bar p(t,T,x,y)$, and we can conclude as in the diagonal case.
Consequently, the problematic case is when $\rho \asymp |y-x|$. 
Indeed, in this case, $\tilde{p}(t,T,x+\rho \theta,y) $ can be in diagonal regime, when $\tilde{p}(t,T,x,y)$ is still in the off-diagonal regime.

Assume first that \textbf{[H-1-a]} holds, that is that $\mu$ has a density with respect to the Lebesgue measure in the sphere: $\mu(d\theta)=f_\mu(\theta)d\theta$.
The remaining integral becomes:
\begin{eqnarray*}
&&\int_{1/2|y-x|}^{3/2|y-x|}  \int_{S^{d-1}}  \tilde{p}(t,T,x+\rho\theta,y) \ind_{\{\rho\ge (T-t)^{1/\alpha} \}} \frac{\bar{q}(\rho)}{\rho^{d+\alpha}} f_\mu(\theta) \rho^{d-1}d\rho d\theta.
\end{eqnarray*}
Now, since $\rho\asymp |y-x|$, we can take $\frac{\bar{q}(\rho)}{\rho^{\alpha+d}}$ out of the integral to get:
\begin{eqnarray*}
&&\int_{1/2|y-x|}^{3/2|y-x|}  \int_{S^{d-1}}  \tilde{p}(t,T,x+\rho \theta ,y) \ind_{\{\rho\ge (T-t)^{1/\alpha} \}} \frac{\bar{q}(\rho)}{\rho^{d+\alpha}} f_\mu(\theta) \rho^{d-1}d\rho d\theta\\
&\le& C\frac{\bar{q}(|y-x|)}{|y-x|^{\alpha+d}}\int_0^{+\infty} \int_{S^{d-1}}  \tilde{p}(t,T,x+\rho \theta,y) \ind_{\{\rho\ge (T-t)^{1/\alpha} \}} f_\mu(\theta) \rho^{d-1}d\rho d\theta.
\end{eqnarray*}


Bounding $f_\mu$ by some constant, we recover the integral of the density $\tilde{p}$ which is lower that one.
Consequently, we have:
\begin{eqnarray*}
\int_{1/2|y-x|}^{3/2|y-x|}  \int_{S^{d-1}}  \tilde{p}(t,T,x+\rho \theta,y) \ind_{\{\rho\ge (T-t)^{1/\alpha} \}} \frac{\bar{q}(\rho)}{\rho^{1+\alpha}}d\rho f_\mu(\theta) d\theta
\le\frac{C}{T-t} \frac{T-t}{|y-x|^{\alpha+d}}\bar{q}(|y-x|) \asymp \frac{C}{T-t} \bar{p}(tT,x,y).
\end{eqnarray*}

Now, assume that \textbf{[H-1-b]} holds.
In this case, we can take out $\frac{\bar{q}(r)}{r^{1+\alpha}}$.
\begin{eqnarray*}
\int_{\frac12|y-x|}^{\frac32|y-x|}  \int_{S^{d-1}}  \tilde{p}(t,T,x+\rho \theta,y) \ind_{\{\rho\ge (T-t)^{\frac{1}{\alpha}} \}} \frac{\bar{q}(\rho)}{\rho^{1+\alpha}}d\rho \mu(d\theta)
\le \frac{\bar{q}(|y-x|)}{|y-x|^{1+\alpha}}\int_0^{+\infty} \int_{S^{d-1}}  \tilde{p}(t,T,x+\rho \theta ,y) \ind_{\{ \rho\ge (T-t)^{\frac{1}{\alpha}} \}} d\rho \mu(d\theta). 
\end{eqnarray*}

Rewriting the right hand side to make the time dependencies  appear : 
\begin{eqnarray*}
\frac{\bar{q}(|y-x|)}{|y-x|^{1+\alpha}}=\frac{1}{T-t}\frac{(T-t)^{1+\frac{\gamma-d}{\alpha}}}{|y-x|^{1+\alpha}}\bar{q}(|y-x|) \times (T-t)^{\frac{d-\gamma}{\alpha}}\le C \frac{1}{T-t}\frac{(T-t)^{1+\frac{\gamma-d}{\alpha}}}{|y-x|^{1+\alpha}}\bar{q}(|y-x|).
\end{eqnarray*}
In the last inequality, we recall that $\gamma \le d$, so that $(T-t)^{\frac{d-\gamma}{\alpha}} \le 1$.
Now, we write:
$$
\frac{(T-t)^{1+\frac{\gamma-d}{\alpha}}}{|y-x|^{1+\alpha}}\bar{q}(|y-x|)
=\frac{(T-t)^{1+\frac{\gamma-d}{\alpha}}}{|y-x|^{\alpha+\gamma}} \times |y-x|^{\gamma-1}\bar{q}(|y-x|).
$$
Define now $Q(|y-x|) = \min(1,|y-x|^{\gamma-1})\bar{q}(|y-x|),$
we finally obtain:
$$
\frac{C}{T-t}\frac{(T-t)^{1+\frac{\gamma-d}{\alpha}}}{|y-x|^{1+\alpha}}\bar{q}(|y-x|) \le \frac{C}{T-t}\frac{(T-t)^{1+\frac{\gamma-d}{\alpha}}}{|y-x|^{\alpha+\gamma}}Q(|y-x|) \le \frac{C}{T-t}\bar{p}(t,T,x,y).
$$

In other words, we can correct the wrong decay by deteriorating the temperation.
Consequently, the global upper bound for the kernel is the one announced.

%
%
%
%
%

%
%
\end{proof}

\section*{Acknowledgments}
This article was prepared within the framework of a subsidy granted to the HSE by the Government of the Russian Federation for the implementation if the Global Competitiveness Program.

\bibliographystyle{alpha}
\bibliography{MyLibrary.bib}

\end{document}